\documentclass[11pt,a4paper]{article}
\usepackage[margin=26mm]{geometry}
\usepackage[T1]{fontenc}
\usepackage{lmodern}
\usepackage{amsmath,amssymb,amsthm,mathtools}
\usepackage{booktabs,longtable,array}
\usepackage{hyperref}
\usepackage{enumitem}
\usepackage{microtype}
\usepackage{tikz}
\usepackage{graphicx}
\usetikzlibrary{arrows.meta,positioning}
\hypersetup{colorlinks=true,linkcolor=blue!50!black,urlcolor=blue!50!black}
\newtheorem{theorem}{Theorem}[section]
\newtheorem{lemma}[theorem]{Lemma}
\newtheorem{proposition}[theorem]{Proposition}
\newtheorem{corollary}[theorem]{Corollary}
\newtheorem{remark}[theorem]{Remark}
\newcommand{\ZZ}{\mathbb Z}
\newcommand{\FF}{\mathbb F_2}
\newcommand{\bd}{\partial}

\newcommand{\PP}{\mathsf P}
\newcommand{\TT}{\mathsf T}

\title{An A--B--C Coordinate Encoding for Exact Boundary-Cycle Extraction on Hexagonal Lattices\\\large Mathematical formulation, proofs, and source-level qualifications}
\author{Philippe Mullhaupt\\Laboratoire d'Automatique, EPFL\\Lausanne, Switzerland}
\date{October 2026}
\begin{document}
\maketitle
\begin{abstract}
We study an integer-coordinate encoding of boundary incidences for finite subsets of the regular hexagonal tiling. Each selected cell contributes six ordered records with endpoints in two role-distinguished classes, point and tail. Identical records from neighboring cells cancel, and the surviving records form a disjoint union of cycles, including the exterior boundaries of disconnected components and the boundaries of holes. We derive the actual even- and odd-row record formulas from the A--B--C transforms used in Hexagones 2.0, prove exact shared-edge identities and the global role-qualified incidence structure, and thereby obtain an unconditional correctness theorem for the \emph{abstract coordinate construction}. We compare the encoding with an independent exact cube-coordinate half-edge hashing implementation, proving an incidence isomorphism and checking 504 finite regions. We give an expected-linear-time indexed algorithm and distinguish it from the historical fixed-capacity C traversal, whose sentinel and closure conventions do not inherit the theorem automatically. The claim is a compact coordinate realization of a known planar boundary principle, not historical priority or a new three-axis lattice.
\end{abstract}

\section{Historical context and scope}
Hexagonal-grid geometry and digital topology have a substantial history. Early studies of hexagonal distance and coordinate descriptions include Luczak and Rosenfeld~\cite{luczak1976}; digital topology and contour definitions were treated by Kovalevsky~\cite{kovalevsky1984} and surveyed by Kong and Rosenfeld~\cite{kong1989}. The cellular-complex viewpoint is standard in algebraic topology~\cite{hatcher2002}, and its use with hexagonal and triangular coordinate systems has been studied explicitly by Nagy~\cite{nagy2015}. A modern practical account of cube and offset coordinates is provided by Patel~\cite{patelhex}.

These references are related at different levels: Luczak and Rosenfeld address hexagonal grid metrics, Kovalevsky addresses discrete contours, Kong and Rosenfeld survey digital topology, and Nagy treats coordinate labels for cells, edges, and vertices. In particular, Nagy's treatment is close in subject matter to the present representation and must not be mistaken for an unrelated application. No claim is made that integer hexagonal coordinates, edge cancellation, boundary cycles, or coordinate labels for vertices and edges are novel in themselves. The limited aim here is to give explicit parity-dependent formulas and a checkable proof for one particular point/tail encoding. We do not establish that this encoding is the first, unique, or more efficient than all earlier encodings.

\section{Problem and scope}
Let $\mathcal H$ be the infinite regular hexagonal tiling, with cells indexed by integer pairs $(x,y)$ in a staggered-row arrangement. Given a finite set $R\subset\ZZ^2$ of selected cells, our objective is to enumerate the full boundary of their union as a collection of closed edge cycles. We allow holes and disconnected selected regions. The computational method is restricted to integer arithmetic: it does not need floating-point intersections or a geometric polygon-union operation.

This note proves the mathematics of the local record construction and an indexed cycle extraction. It does \emph{not} assert that every detail of the original C function is correct for arbitrary inputs; Section~\ref{sec:implementation} makes this distinction explicit. The three-number coordinates, cellular boundary operator, and degree-two cycle decomposition are established concepts; no priority claim is made.

\section{A--B--C coordinates and integer conversion}
Define $(A,B,C)\sim(A+k,B+k,C+k)$ for $k\in\ZZ$. The quotient $\ZZ^3/\langle(1,1,1)\rangle$ is a rank-two lattice. Write
\begin{equation}\label{eq:normalize}
 N(A,B,C)=(A-m,B-m,C-m),\qquad m=\min(A,B,C).
\end{equation}
\begin{proposition}[Canonical normalization]
Every equivalence class contains precisely one triple with minimum coordinate zero.
\end{proposition}
\begin{proof}
Subtracting the minimum constructs one. If two such representatives differ by $(k,k,k)$, their minima differ by $k$; hence $k=0$.
\end{proof}

The source conversion from primal indices to A--B--C coordinates is
\begin{equation}\label{eq:forward}
 A=x+1+\left\lfloor\frac{y+1}{2}\right\rfloor,\quad B=y+1,\quad C=0.
\end{equation}
For a normalized triple $(a,b,c)$, the reverse conversion is
\begin{equation}\label{eq:reverse}
 X=a-\left\lfloor\frac{b+c}{2}\right\rfloor-1,\qquad Y=b-c-1.
\end{equation}
The C source uses integer division $(b+c)/2$ rather than a call to $\lfloor\cdot\rfloor$; after normalization $b+c\ge0$, so the expressions agree. Equation~\eqref{eq:forward} uses the mathematical floor of $(y+1)/2$; the source actually evaluates \texttt{floor(y/2)} after incrementing \texttt{y}. Since \texttt{y/2} is integer division in C, this agrees with the mathematical floor only for nonnegative incremented $y$. Accordingly, the literal source-level derivations in this note apply to the ordinary nonnegative map-index domain, with $y\ge -1$ sufficient for the forward division and additional bounds as appropriate. The algebraic extension of \eqref{eq:forward} to all $\ZZ^2$ is a mathematical extension, not a claim about negative C inputs.

For comparison, $(q,r,s)=(A-C,B-A,C-B)$ has $q+r+s=0$, so this lattice is isomorphic to standard cube coordinates. We make no novelty claim about the lattice itself.

\section{Six coordinate records per cell}
For $u=(A,B,C)$ define
\begin{equation}\label{eq:aux}
 \begin{aligned}U_0&=(A,B,C), & U_1&=(A+1,B+1,C),\\ U_2&=(A+1,B,C+1), & U_3&=(A,B+1,C),\\ U_4&=(A,B,C+1).\end{aligned}
\end{equation}
Apply the normalized reverse map~\eqref{eq:reverse} to each triple and write $D(U_i)$ for the resulting integer pair. The six records are
\begin{equation}\label{eq:records}
\begin{array}{lll}
 e_1=(D(U_2),D(U_1)),&e_2=(D(U_0),D(U_1)),&e_3=(D(U_0),D(U_3)),\\
 e_4=(D(U_4),D(U_3)),&e_5=(D(U_4),D(U_0)),&e_6=(D(U_2),D(U_0)).
\end{array}
\end{equation}
The first entry is a \emph{point} identifier, the second a \emph{tail} identifier. Their roles must be kept distinct: $(\PP,u,v)\ne(\TT,u,v)$ even if the displayed integer pairs coincide. Thus each record is an edge of a bipartite incidence graph, not a segment between two unqualified points in the ordinary $(X,Y)$ plane.

\begin{lemma}[Single-cell six-cycle]\label{lem:cellcycle}
For every cell, the six records in~\eqref{eq:records} form a simple six-cycle on the six role-qualified vertices $\PP D(U_2),\PP D(U_0),\PP D(U_4)$ and $\TT D(U_1),\TT D(U_3),\TT D(U_0)$.
\end{lemma}
\begin{proof}
The cyclic order of the edges is $e_1,e_2,e_3,e_4,e_5,e_6$: successive pairs share alternately a tail and a point identifier, with $e_6$ returning to $e_1$ at the point identifier $D(U_2)$. The parity formulas of Appendix~\ref{app:formulas} show that the three point identifiers are distinct and the three tail identifiers are distinct, for either row parity. Hence the six role-qualified vertices are distinct and the walk is a simple six-cycle.
\end{proof}

\section{Exact neighbor matching}
For the remainder, let $e_i(x,y)$ denote the record from cell $(x,y)$. We use the notation $e_i(x,y)=e_j(x',y')$ for equality of both ordered integer endpoint pairs.

\begin{theorem}[Six-neighbor record identities]\label{thm:neighbors}
The six records of a cell $(x,y)$ are each shared with exactly one other cell. Their partners are as follows:
\begin{center}\small
\begin{tabular}{ccccc}\toprule
Central edge&Neighbor edge&Even $y$: $(\Delta x,\Delta y)$&Odd $y$: $(\Delta x,\Delta y)$\\\midrule
$1$&$4$&$(1,0)$&$(1,0)$\\
$2$&$5$&$(0,1)$&$(1,1)$\\
$3$&$6$&$(-1,1)$&$(0,1)$\\
$4$&$1$&$(-1,0)$&$(-1,0)$\\
$5$&$2$&$(-1,-1)$&$(0,-1)$\\
$6$&$3$&$(0,-1)$&$(1,-1)$\\\bottomrule
\end{tabular}
\end{center}
No third cell generates the same ordered record.
\end{theorem}
\begin{proof}
The exact even/odd formulas in Tables~\ref{tab:even} and~\ref{tab:odd} are obtained by substitution in~\eqref{eq:forward}--\eqref{eq:records}. Substituting each offset in the indicated neighboring parity gives identical first and second endpoints; all twelve cases can be checked term by term. For example, with even $y$,
\[
 e_2(x,y)=((x,y),(x,y+1))=e_5(x,y+1),
\]
where $y+1$ is odd. With odd $y$,
\[
 e_2(x,y)=((x,y),(x+1,y+1))=e_5(x+1,y+1),
\]
where $y+1$ is even.

To establish uniqueness, fix a record $e=(\PP(u,v),\TT(s,t))$. The role-qualified endpoint-incidence lists in Appendix~\ref{app:vertices} give exactly three candidate cells for the point endpoint. Among their six local records, the endpoint $\TT(s,t)$ occurs on exactly two cells, as the listed formulas show. Equivalently, at a point-role vertex the three incident edge types are pairwise distinct; every such type is generated by exactly two of its three cells. Thus no third cell produces the record. This argument applies to every integer translate and both parities.
\end{proof}

\section{Global endpoint incidence}
The local formulas also identify the three primal cells meeting at each role-qualified vertex. Appendix~\ref{app:vertices} lists all four parity/role cases, with the exact local edge numbers.

\begin{theorem}[Role-qualified incidence correspondence]\label{thm:incidence}
The bipartite graph generated by all records $e_i(x,y)$, with identical ordered records identified, is combinatorially isomorphic to the vertex--edge graph of the regular hexagonal tiling. Each role-qualified vertex is incident to three distinct full-lattice edges and three cells; every edge is incident to two cells.
\end{theorem}
\begin{proof}
Lemma~\ref{lem:cellcycle} gives one simple six-cycle for each cell. Appendix~\ref{app:vertices} shows that every role-qualified endpoint is used by exactly three distinct cells, with two consecutive edges contributed by each cell. Theorem~\ref{thm:neighbors} identifies the same edge in exactly two cells, and its endpoint lists show the three incident full-lattice edges at each vertex.

These incidence data give a locally hexagonal, three-valent cellulation. To identify it with the intended \emph{global} tiling, note that the cell labels $(x,y)$ and their six edge-neighbor offsets are precisely the even/odd staggered-row adjacency relations for the infinite hexagonal lattice. Map each abstract cell to the standard staggered-row hexagon with that label. Map each record to the unique common side of the two cells in Theorem~\ref{thm:neighbors}. At a cell, its six records correspond bijectively to the six distinct neighbors; at a vertex, the three cells in Appendix~\ref{app:vertices} are exactly the three cells meeting at the corresponding hexagon corner. Thus the map preserves and reflects cell--edge and edge--vertex incidence and is bijective on cells, edges and vertices. This is the desired combinatorial isomorphism.
\end{proof}

\begin{remark}
The proof concerns \emph{role-qualified} endpoints. If $\PP(u,v)$ and $\TT(u,v)$ were identified, a single cell would have a degree-four numerical vertex and the theorem would fail. Alternating point/tail matching is therefore structural, not merely a coding convenience.
\end{remark}

\section{Boundary cancellation and cycle decomposition}
For a finite set $R$ of cells, form the multiset of its $6|R|$ records. Cancel every identical pair and call the set of survivors $E_\bd(R)$.

\begin{theorem}[Exact cancellation]\label{thm:cancel}
A record survives if and only if its corresponding geometric edge has exactly one incident cell in $R$.
\end{theorem}
\begin{proof}
By Theorems~\ref{thm:neighbors} and~\ref{thm:incidence}, every geometric edge has exactly two incident cells and the two cells generate the same record. If neither is selected, the record occurs zero times; if one is selected, once; if both are selected, twice. Pair cancellation therefore retains exactly the edges with one selected incident cell.
\end{proof}

\begin{theorem}[Degree-two boundary graph]\label{thm:degree}
Every vertex of the surviving role-qualified boundary graph has degree zero or two.
\end{theorem}
\begin{proof}
Fix a vertex. Exactly three cells $f_1,f_2,f_3$ meet there by Theorem~\ref{thm:incidence}. Write $a_i\in\{0,1\}$ for their selection indicators. The three incident edges separate the pairs $(f_1,f_2)$, $(f_2,f_3)$, $(f_3,f_1)$. By Theorem~\ref{thm:cancel}, the number of surviving edges is
\[
 (a_1\oplus a_2)+(a_2\oplus a_3)+(a_3\oplus a_1),
\]
which is zero if all three values agree and two otherwise.
\end{proof}

\begin{theorem}[Complete closed-cycle decomposition]\label{thm:cycles}
For every finite cell set $R$, the surviving coordinate boundary graph is a disjoint union of simple cycles. Each boundary edge belongs to exactly one cycle. No hypothesis of connectedness or simple connectedness is required.
\end{theorem}
\begin{proof}
The surviving graph is finite. By Theorem~\ref{thm:degree}, every non-isolated vertex has degree two. Each connected component of a finite degree-two graph is a cycle, and components are edge-disjoint. Theorem~\ref{thm:cancel} identifies the union of these cycles with the complete geometric boundary.
\end{proof}

\begin{corollary}[Holes and disconnected components]\label{cor:holes}
Let $c$ be the number of connected components of the union of selected closed hexagons, $h$ the number of bounded connected components of its complement, and $m$ the number of boundary cycles. Then $m=c+h$.
\end{corollary}
\begin{proof}
Theorem~\ref{thm:cycles} shows that the boundary consists of disjoint Jordan cycles. The selected union is a compact planar surface with boundary: the three-cells-per-vertex local configurations produce no pinched boundary vertices. Each connected component has one outer boundary and one additional boundary for each hole contained in that component. Summing over components gives $m=c+h$. Equivalently, $\chi=c-h$ and the total number of boundary components of a compact planar surface is $c+h$.
\end{proof}

\section{Indexed extraction and complexity}
Let $n=|R|$ and $b=|E_\bd(R)|\le 6n$. An indexed realization is:
\begin{enumerate}[label=\arabic*.]
\item For each cell, generate its six ordered integer records.
\item Maintain a dictionary keyed by all four endpoint integers. Insert an absent record and delete an existing one (symmetric difference).
\item For each surviving record, add its edge identifier to the incidence lists of its two \emph{role-qualified} endpoints.
\item While unused edges remain, choose one and follow the unique other incident edge at each vertex until the starting vertex is reached. Output the cycle and mark its edges used.
\end{enumerate}
\begin{theorem}[Indexed extraction]\label{thm:algorithm}
The algorithm terminates and outputs each boundary cycle exactly once, up to cyclic rotation and reversal. With expected constant-time dictionary operations, it uses expected $O(n+b)=O(n)$ time and $O(n)$ space.
\end{theorem}
\begin{proof}
Theorem~\ref{thm:cancel} proves correctness of the dictionary stage. Theorem~\ref{thm:cycles} makes each connected component a simple cycle, so continuation through the other incident edge is unique and returns to the start after traversing exactly that component. Marking edges used prevents repetition and exhausts all components. Generation and dictionary operations take expected $O(n)$ time, incidence construction $O(b)$, and each boundary edge is traversed once, $O(b)$. The stored records and indices use $O(n)$ space.
\end{proof}

The expected linear bound is optimal in the standard explicit-input/output model, since $n$ cell labels must be read and $b$ boundary records output. Sorting records instead gives deterministic $O(n\log n)$ cancellation time. The original source uses sorting and repeated scans for continuation; it is \emph{not} an implementation of the indexed linear-time algorithm.

\section{Relation to existing mathematics}
The equivalence class of three-integer coordinates is a form of the conventional hexagonal lattice. Pair cancellation is the cellular boundary operator $\bd_2$ with coefficients in $\FF$; the boundary graph is degree two because three hexagons meet at a vertex. In algebraic-topological notation, the method realizes $\bd_2$ by concrete integer record identities. The point/tail role distinction is a coordinate realization of the bipartition of the hexagonal vertex graph. None of these general principles is claimed to be new.

The narrow technical contribution documented here is the particular A--B--C-to-record map~\eqref{eq:records}, its parity-explicit incidence identities, and its exact boundary-cycle interpretation. Claims of lower memory usage or speed relative to a well-engineered half-edge implementation require measurements and are not made here.

\section{Original C routine: scope of the mathematical proof}\label{sec:implementation}
The relevant Increment 38 source files contain the coordinate transformations and boundary-extraction routine. The latter uses the six formulas~\eqref{eq:records}, sorts the records, removes equal pairs, and searches for matching endpoints while alternating endpoint roles. It has implementation-specific limitations:
\begin{itemize}
\item At most ten returned loops and a fixed \texttt{maxSquaresForSuperpath} path capacity; capacity failure returns zero paths.
\item The value \texttt{2000} is used as a processed-coordinate sentinel, so inputs using this coordinate need special care.
\item Closure compares the stored \emph{numeric} coordinates of the current and starting positions; it does not explicitly compare role tags. The abstract proof requires role-qualified equality.
\item The duplicate-removal and end-of-array traversal logic require their own boundary-case audit, separate from the algebraic cancellation theorem.
\item The original continuation searches the array linearly and can require quadratic time in the number of boundary records.
\end{itemize}
Accordingly, Theorems~\ref{thm:incidence}--\ref{thm:algorithm} prove the mathematical construction and the stated indexed algorithm, not unrestricted behavior of this fixed-capacity C routine. No changes to game code are proposed here.

\section{Comparison with cube-coordinate half-edge hashing}\label{sec:comparison}
The A--B--C representation should not be compared with a straw-man
floating-point polygon union. A fair baseline is an integer-only
half-edge representation on the same tiling. We distinguish (i) the
coordinate system for cell \emph{centers} from (ii) the identifiers for
geometric \emph{edges and vertices}. The latter is where the present
encoding is particularly concise.

\subsection{Flattened cube coordinates and incidence identifiers}
For the staggered-row cell $(x,y)$, define the conventional axial/cube
coordinates
\begin{equation}\label{eq:cubemap}
 q=x-\left\lfloor\frac{y-(y\bmod2)}2\right\rfloor,\qquad
 r=y,\qquad s=-q-r.
\end{equation}
Here $y\bmod2\in\{0,1\}$ even for negative $y$. Only two coordinates
are independent. The map is a bijection from staggered-row cell labels
to the integer cube plane $q+r+s=0$.

A cube coordinate labels a \emph{cell center}; it does not, by itself,
identify a corner of a hexagonal cell. A direct, exact, geometry-free
baseline identifies a corner by the \emph{unordered triple of cell
centers} incident to it. Let $N_i(f)$ denote the neighboring cell
across local edge $e_i$ in the order $1,\ldots,6$, and let $Q(f)$ be
its cube coordinate. Define
\begin{equation}\label{eq:cubevertex}
 V_{ij}(f)=\operatorname{sort}\{Q(f),Q(N_i(f)),Q(N_j(f))\},
\end{equation}
for cyclically adjacent sides $i,j$. The sorted triple is invariant
under which of the three incident cells generates it. The
corresponding oriented half-edge record for local side $i$ is
\begin{equation}\label{eq:cubehalf}
 h_i(f)=\bigl(V_{i,p(i)}(f),V_{i,t(i)}(f)\bigr),
\end{equation}
where the local point/tail adjacent-side indices are
\begin{equation}\label{eq:adjtable}
\begin{array}{c|rrrrrr}
 i&1&2&3&4&5&6\\\hline
 p(i)&6&3&2&5&4&1\\
 t(i)&2&1&4&3&6&5
\end{array}.
\end{equation}
The notation ``half-edge'' here means an oriented incidence record
with ordered endpoints, not necessarily a full DCEL with explicit
\texttt{twin}, \texttt{next}, and face pointers. Shared sides produce
identical canonical records; if instead one stores opposite directed
half-edges, one must canonicalize before cancellation.

\begin{proposition}[Equivalence of exact incidence encodings]
\label{prop:comparison}
The point/tail edge records of~\eqref{eq:records} and the
cube-triple records of~\eqref{eq:cubehalf} encode isomorphic
cell--edge--vertex incidences. Consequently they yield identical
boundary cycle decompositions for every finite cell set.
\end{proposition}
\begin{proof}
Theorem~\ref{thm:incidence} identifies each role-qualified A--B--C
vertex with exactly one geometric corner, incident to exactly three
cells. Equation~\eqref{eq:cubevertex} identifies that same corner by
its three incident cell centers. Distinct corners cannot have the
same incident-cell triple: in the regular hexagonal tiling three
mutually adjacent cells meet at a unique corner. Thus the assignment
from an A--B--C role-qualified identifier to the sorted cube triple
is well-defined and bijective. The two endpoints of each local edge
are its two corners; the adjacency table~\eqref{eq:adjtable} selects
the point and tail corners in the same order as~\eqref{eq:records}.
The map therefore preserves edge incidence and record equality.
Cancellation and cycle extraction depend only on these relations,
so the resulting cycle decompositions agree.
\end{proof}

\subsection{Representation cost and computational trade-offs}
A normalized A--B--C triple contains three nonnegative integers but
has only two degrees of freedom. Its cell labels are not intrinsically
more information-efficient than cube or axial coordinates. The
additional benefit is that its local arithmetic supplies \emph{edge
endpoint identifiers} without explicitly storing the three incident
cell centers. In the present implementation, an A--B--C edge key
consists of two role-qualified integer pairs, i.e. four integer
components plus implicit point/tail roles. The cube-triple baseline
uses two endpoint keys, each containing three cube centers (nine
integer components before exploiting the sum-zero relation or
canonical compression). This is a comparison of \emph{uncompressed
logical keys}, not of optimized memory footprints: a conventional
integer vertex coordinate, lattice embedding, or a compact DCEL
can eliminate much of the baseline's redundancy.

Both constructions generate six records per cell, hash canonical
edges, index surviving edges by endpoints, and traverse degree-two
components. With expected constant-time dictionary operations, both
have expected $O(n+b)=O(n)$ time and $O(n)$ space, where $n$ is the
number of selected cells and $b$ the number of boundary edges.
The point/tail construction does not improve this asymptotic bound.
Its identifiable merit is an explicit local integer labeling of
vertices and edges without a separate vertex construction.

\subsection{Reproducible implementation comparison}
The companion program
\path{verification/compare_encodings.py} implements both encodings
independently, checks a global bijection between their vertex keys
on a finite interior window, compares exact boundary edge sets under
that bijection, and compares sorted cycle-length multisets for 504
finite cell sets. All these checks passed in the supplied run. A
small optional Python timing experiment is provided as a diagnostic,
not as a performance claim: the cube-triple baseline deliberately
uses large nested tuple keys and is not a memory-optimized or
compiled half-edge implementation. Runtime ratios therefore do not
establish an intrinsic advantage of A--B--C over optimized cube or
half-edge schemes.

\begin{remark}[What is and is not established]
The formal equivalence proof is combinatorial and applies to all
finite cell sets. The accompanying executable tests are finite
regression checks, not the proof. Neither the equivalence nor the
timings demonstrate historical novelty, minimal encoding size, or
universal superiority over existing representations.
\end{remark}

\section{Conclusion}
The A--B--C point/tail formulas give a faithful integer-coordinate encoding of the regular hexagonal cell complex. Exact record cancellation produces precisely the selected region's boundary, whose degree-two graph decomposes into all closed boundary cycles, with no need to classify holes in advance. An indexed implementation has expected linear complexity. The result is a focused coordinate realization and correctness analysis, not a claim of historical priority.

\section{Coordinate and boundary illustrations}
The eight diagrams below are drawn directly in TikZ and remain fully scalable. They are schematic illustrations of the incidence construction; the parity-specific algebraic identities are established in Appendices A--D.
\begingroup\centering\setlength{\tabcolsep}{4pt}
\begin{minipage}[t]{0.48\linewidth}
\centering
\resizebox{0.96\linewidth}{!}{\begin{tikzpicture}[x=0.85cm,y=0.85cm,baseline=(current bounding box.center)]
\path[draw=gray!45,fill=white,line width=.45pt] (-1.634,-0.470) -- (-2.050,-0.230) -- (-2.466,-0.470) -- (-2.466,-0.950) -- (-2.050,-1.190) -- (-1.634,-0.950) -- cycle;
\path[draw=gray!45,fill=white,line width=.45pt] (-1.224,0.240) -- (-1.640,0.480) -- (-2.056,0.240) -- (-2.056,-0.240) -- (-1.640,-0.480) -- (-1.224,-0.240) -- cycle;
\path[draw=gray!45,fill=white,line width=.45pt] (-0.814,0.950) -- (-1.230,1.190) -- (-1.646,0.950) -- (-1.646,0.470) -- (-1.230,0.230) -- (-0.814,0.470) -- cycle;
\path[draw=gray!45,fill=white,line width=.45pt] (-0.814,-0.470) -- (-1.230,-0.230) -- (-1.646,-0.470) -- (-1.646,-0.950) -- (-1.230,-1.190) -- (-0.814,-0.950) -- cycle;
\path[draw=gray!45,fill=white,line width=.45pt] (-0.404,0.240) -- (-0.820,0.480) -- (-1.236,0.240) -- (-1.236,-0.240) -- (-0.820,-0.480) -- (-0.404,-0.240) -- cycle;
\path[draw=gray!45,fill=white,line width=.45pt] (0.006,0.950) -- (-0.410,1.190) -- (-0.826,0.950) -- (-0.826,0.470) -- (-0.410,0.230) -- (0.006,0.470) -- cycle;
\path[draw=gray!45,fill=white,line width=.45pt] (0.006,-0.470) -- (-0.410,-0.230) -- (-0.826,-0.470) -- (-0.826,-0.950) -- (-0.410,-1.190) -- (0.006,-0.950) -- cycle;
\path[draw=gray!45,fill=cyan!17,line width=.45pt] (0.416,0.240) -- (0.000,0.480) -- (-0.416,0.240) -- (-0.416,-0.240) -- (-0.000,-0.480) -- (0.416,-0.240) -- cycle;
\path[draw=gray!45,fill=white,line width=.45pt] (0.826,0.950) -- (0.410,1.190) -- (-0.006,0.950) -- (-0.006,0.470) -- (0.410,0.230) -- (0.826,0.470) -- cycle;
\path[draw=gray!45,fill=white,line width=.45pt] (0.826,-0.470) -- (0.410,-0.230) -- (-0.006,-0.470) -- (-0.006,-0.950) -- (0.410,-1.190) -- (0.826,-0.950) -- cycle;
\path[draw=gray!45,fill=white,line width=.45pt] (1.236,0.240) -- (0.820,0.480) -- (0.404,0.240) -- (0.404,-0.240) -- (0.820,-0.480) -- (1.236,-0.240) -- cycle;
\path[draw=gray!45,fill=white,line width=.45pt] (1.646,0.950) -- (1.230,1.190) -- (0.814,0.950) -- (0.814,0.470) -- (1.230,0.230) -- (1.646,0.470) -- cycle;
\path[draw=gray!45,fill=white,line width=.45pt] (1.646,-0.470) -- (1.230,-0.230) -- (0.814,-0.470) -- (0.814,-0.950) -- (1.230,-1.190) -- (1.646,-0.950) -- cycle;
\path[draw=gray!45,fill=white,line width=.45pt] (2.056,0.240) -- (1.640,0.480) -- (1.224,0.240) -- (1.224,-0.240) -- (1.640,-0.480) -- (2.056,-0.240) -- cycle;
\path[draw=gray!45,fill=white,line width=.45pt] (2.466,0.950) -- (2.050,1.190) -- (1.634,0.950) -- (1.634,0.470) -- (2.050,0.230) -- (2.466,0.470) -- cycle;\node[font=\scriptsize] at (0,0) {$f$};
\end{tikzpicture}}
\par\vspace{2pt}{\footnotesize\textbf{Fig. 1.} Hexagonal cells and their local neighborhood.\par}
\end{minipage}
\hfill
\begin{minipage}[t]{0.48\linewidth}
\centering
\resizebox{0.96\linewidth}{!}{\begin{tikzpicture}[x=0.85cm,y=0.85cm,baseline=(current bounding box.center)]
\path[draw=gray!45,fill=white,line width=.45pt] (-0.814,-0.470) -- (-1.230,-0.230) -- (-1.646,-0.470) -- (-1.646,-0.950) -- (-1.230,-1.190) -- (-0.814,-0.950) -- cycle;
\path[draw=gray!45,fill=white,line width=.45pt] (-0.404,0.240) -- (-0.820,0.480) -- (-1.236,0.240) -- (-1.236,-0.240) -- (-0.820,-0.480) -- (-0.404,-0.240) -- cycle;
\path[draw=gray!45,fill=white,line width=.45pt] (0.006,0.950) -- (-0.410,1.190) -- (-0.826,0.950) -- (-0.826,0.470) -- (-0.410,0.230) -- (0.006,0.470) -- cycle;
\path[draw=gray!45,fill=white,line width=.45pt] (0.006,-0.470) -- (-0.410,-0.230) -- (-0.826,-0.470) -- (-0.826,-0.950) -- (-0.410,-1.190) -- (0.006,-0.950) -- cycle;
\path[draw=gray!45,fill=cyan!17,line width=.45pt] (0.416,0.240) -- (0.000,0.480) -- (-0.416,0.240) -- (-0.416,-0.240) -- (-0.000,-0.480) -- (0.416,-0.240) -- cycle;
\path[draw=gray!45,fill=white,line width=.45pt] (0.826,0.950) -- (0.410,1.190) -- (-0.006,0.950) -- (-0.006,0.470) -- (0.410,0.230) -- (0.826,0.470) -- cycle;
\path[draw=gray!45,fill=white,line width=.45pt] (0.826,-0.470) -- (0.410,-0.230) -- (-0.006,-0.470) -- (-0.006,-0.950) -- (0.410,-1.190) -- (0.826,-0.950) -- cycle;
\path[draw=gray!45,fill=white,line width=.45pt] (1.236,0.240) -- (0.820,0.480) -- (0.404,0.240) -- (0.404,-0.240) -- (0.820,-0.480) -- (1.236,-0.240) -- cycle;
\path[draw=gray!45,fill=white,line width=.45pt] (1.646,0.950) -- (1.230,1.190) -- (0.814,0.950) -- (0.814,0.470) -- (1.230,0.230) -- (1.646,0.470) -- cycle;\node[font=\scriptsize] at (0,0) {$(0,0,0)$};\node[font=\scriptsize] at (0,1.25) {$N(A,B,C)=(A-m,B-m,C-m)$};\node[font=\scriptsize] at (0,-1.35) {$m=\min(A,B,C)$};
\end{tikzpicture}}
\par\vspace{2pt}{\footnotesize\textbf{Fig. 2.} Canonical three-integer coordinate normalization.\par}
\end{minipage}
\par\vspace{9pt}
\begin{minipage}[t]{0.48\linewidth}
\centering
\resizebox{0.96\linewidth}{!}{\begin{tikzpicture}[x=0.85cm,y=0.85cm,baseline=(current bounding box.center)]
\draw[thick,gray!75] (0.000,1.350) -- (1.169,0.675);
\draw[thick,gray!75] (1.169,0.675) -- (1.169,-0.675);
\draw[thick,gray!75] (1.169,-0.675) -- (0.000,-1.350);
\draw[thick,gray!75] (0.000,-1.350) -- (-1.169,-0.675);
\draw[thick,gray!75] (-1.169,-0.675) -- (-1.169,0.675);
\draw[thick,gray!75] (-1.169,0.675) -- (0.000,1.350);
\node[circle,fill=blue!70,minimum size=5pt,inner sep=0pt] at (0.000,1.350) {};\node[font=\scriptsize] at (0.000,1.688) {$ P_1 $};\node[circle,fill=red!75,minimum size=5pt,inner sep=0pt] at (1.169,0.675) {};\node[font=\scriptsize] at (1.461,0.844) {$ T_2 $};\node[circle,fill=blue!70,minimum size=5pt,inner sep=0pt] at (1.169,-0.675) {};\node[font=\scriptsize] at (1.461,-0.844) {$ P_3 $};\node[circle,fill=red!75,minimum size=5pt,inner sep=0pt] at (0.000,-1.350) {};\node[font=\scriptsize] at (0.000,-1.688) {$ T_4 $};\node[circle,fill=blue!70,minimum size=5pt,inner sep=0pt] at (-1.169,-0.675) {};\node[font=\scriptsize] at (-1.461,-0.844) {$ P_5 $};\node[circle,fill=red!75,minimum size=5pt,inner sep=0pt] at (-1.169,0.675) {};\node[font=\scriptsize] at (-1.461,0.844) {$ T_6 $};
\end{tikzpicture}}
\par\vspace{2pt}{\footnotesize\textbf{Fig. 3.} Six alternating point/tail incidence vertices.\par}
\end{minipage}
\hfill
\begin{minipage}[t]{0.48\linewidth}
\centering
\resizebox{0.96\linewidth}{!}{\begin{tikzpicture}[x=0.85cm,y=0.85cm,baseline=(current bounding box.center)]
\node[circle,fill=blue!70,minimum size=7pt,inner sep=0pt] (p) at (-1,0) {};\node[circle,fill=red!75,minimum size=7pt,inner sep=0pt] (t) at (1,0) {};\node[font=\small,below=6pt] at (p) {$P(u)$};\node[font=\small,below=6pt] at (t) {$T(u)$};\node[font=\scriptsize] at (0,1) {same numeric triple $u$};\node[font=\scriptsize] at (0,-1) {$P(u)\ne T(u)$};
\end{tikzpicture}}
\par\vspace{2pt}{\footnotesize\textbf{Fig. 4.} Role qualification prevents accidental identification.\par}
\end{minipage}
\par\vspace{9pt}
\begin{minipage}[t]{0.48\linewidth}
\centering
\resizebox{0.96\linewidth}{!}{\begin{tikzpicture}[x=0.85cm,y=0.85cm,baseline=(current bounding box.center)]
\path[draw=gray!45,fill=cyan!20,line width=.45pt] (-0.001,0.415) -- (-0.720,0.830) -- (-1.439,0.415) -- (-1.439,-0.415) -- (-0.720,-0.830) -- (-0.001,-0.415) -- cycle;\path[draw=gray!45,fill=orange!20,line width=.45pt] (1.439,0.415) -- (0.720,0.830) -- (0.001,0.415) -- (0.001,-0.415) -- (0.720,-0.830) -- (1.439,-0.415) -- cycle;\draw[red!80!black,very thick] (0,-.415) -- (0,.415);\node[font=\scriptsize] at (-.72,0) {$f$};\node[font=\scriptsize] at (.72,0) {$g$};\node[font=\scriptsize] at (0,-1.2) {$e_f=e_g\ \Rightarrow\ e_f\oplus e_g=0$};
\end{tikzpicture}}
\par\vspace{2pt}{\footnotesize\textbf{Fig. 5.} Two neighboring cells cancel their common edge.\par}
\end{minipage}
\hfill
\begin{minipage}[t]{0.48\linewidth}
\centering
\resizebox{0.96\linewidth}{!}{\begin{tikzpicture}[x=0.85cm,y=0.85cm,baseline=(current bounding box.center)]
\path[draw=gray!45,fill=cyan!17,line width=.45pt] (0.675,0.780) -- (0.338,0.975) -- (0.000,0.780) -- (0.000,0.390) -- (0.338,0.195) -- (0.675,0.390) -- cycle;
\path[draw=gray!45,fill=cyan!17,line width=.45pt] (1.351,-0.390) -- (1.013,-0.195) -- (0.675,-0.390) -- (0.675,-0.780) -- (1.013,-0.975) -- (1.351,-0.780) -- cycle;
\path[draw=gray!45,fill=cyan!17,line width=.45pt] (-0.675,-0.390) -- (-1.013,-0.195) -- (-1.351,-0.390) -- (-1.351,-0.780) -- (-1.013,-0.975) -- (-0.675,-0.780) -- cycle;
\path[draw=gray!45,fill=cyan!17,line width=.45pt] (-1.351,-0.390) -- (-1.689,-0.195) -- (-2.026,-0.390) -- (-2.026,-0.780) -- (-1.689,-0.975) -- (-1.351,-0.780) -- cycle;
\path[draw=gray!45,fill=cyan!17,line width=.45pt] (0.338,0.195) -- (0.000,0.390) -- (-0.338,0.195) -- (-0.338,-0.195) -- (-0.000,-0.390) -- (0.338,-0.195) -- cycle;
\path[draw=gray!45,fill=cyan!17,line width=.45pt] (2.026,0.780) -- (1.689,0.975) -- (1.351,0.780) -- (1.351,0.390) -- (1.689,0.195) -- (2.026,0.390) -- cycle;
\path[draw=gray!45,fill=cyan!17,line width=.45pt] (-0.675,0.780) -- (-1.013,0.975) -- (-1.351,0.780) -- (-1.351,0.390) -- (-1.013,0.195) -- (-0.675,0.390) -- cycle;
\path[draw=gray!45,fill=cyan!17,line width=.45pt] (0.000,0.780) -- (-0.338,0.975) -- (-0.675,0.780) -- (-0.675,0.390) -- (-0.338,0.195) -- (-0.000,0.390) -- cycle;
\path[draw=gray!45,fill=cyan!17,line width=.45pt] (1.689,0.195) -- (1.351,0.390) -- (1.013,0.195) -- (1.013,-0.195) -- (1.351,-0.390) -- (1.689,-0.195) -- cycle;
\path[draw=gray!45,fill=cyan!17,line width=.45pt] (0.675,-0.390) -- (0.338,-0.195) -- (0.000,-0.390) -- (0.000,-0.780) -- (0.338,-0.975) -- (0.675,-0.780) -- cycle;
\path[draw=gray!45,fill=cyan!17,line width=.45pt] (-1.013,0.195) -- (-1.351,0.390) -- (-1.689,0.195) -- (-1.689,-0.195) -- (-1.351,-0.390) -- (-1.013,-0.195) -- cycle;
\path[draw=gray!45,fill=cyan!17,line width=.45pt] (-0.338,0.195) -- (-0.675,0.390) -- (-1.013,0.195) -- (-1.013,-0.195) -- (-0.675,-0.390) -- (-0.338,-0.195) -- cycle;
\path[draw=gray!45,fill=cyan!17,line width=.45pt] (0.000,-0.390) -- (-0.338,-0.195) -- (-0.675,-0.390) -- (-0.675,-0.780) -- (-0.338,-0.975) -- (-0.000,-0.780) -- cycle;
\draw[black,line width=1.3pt] (0.675,0.780) -- (0.338,0.975);
\draw[black,line width=1.3pt] (0.338,0.975) -- (0.000,0.780);
\draw[black,line width=1.3pt] (0.338,0.195) -- (0.675,0.390);
\draw[black,line width=1.3pt] (0.675,0.390) -- (0.675,0.780);
\draw[black,line width=1.3pt] (1.013,-0.195) -- (0.675,-0.390);
\draw[black,line width=1.3pt] (0.675,-0.780) -- (1.013,-0.975);
\draw[black,line width=1.3pt] (1.013,-0.975) -- (1.351,-0.780);
\draw[black,line width=1.3pt] (1.351,-0.780) -- (1.351,-0.390);
\draw[black,line width=1.3pt] (-1.351,-0.780) -- (-1.013,-0.975);
\draw[black,line width=1.3pt] (-1.013,-0.975) -- (-0.675,-0.780);
\draw[black,line width=1.3pt] (-1.689,-0.195) -- (-2.026,-0.390);
\draw[black,line width=1.3pt] (-2.026,-0.390) -- (-2.026,-0.780);
\draw[black,line width=1.3pt] (-2.026,-0.780) -- (-1.689,-0.975);
\draw[black,line width=1.3pt] (-1.689,-0.975) -- (-1.351,-0.780);
\draw[black,line width=1.3pt] (0.338,-0.195) -- (0.338,0.195);
\draw[black,line width=1.3pt] (2.026,0.780) -- (1.689,0.975);
\draw[black,line width=1.3pt] (1.689,0.975) -- (1.351,0.780);
\draw[black,line width=1.3pt] (1.351,0.780) -- (1.351,0.390);
\draw[black,line width=1.3pt] (1.689,0.195) -- (2.026,0.390);
\draw[black,line width=1.3pt] (2.026,0.390) -- (2.026,0.780);
\draw[black,line width=1.3pt] (-0.675,0.780) -- (-1.013,0.975);
\draw[black,line width=1.3pt] (-1.013,0.975) -- (-1.351,0.780);
\draw[black,line width=1.3pt] (-1.351,0.780) -- (-1.351,0.390);
\draw[black,line width=1.3pt] (0.000,0.780) -- (-0.338,0.975);
\draw[black,line width=1.3pt] (-0.338,0.975) -- (-0.675,0.780);
\draw[black,line width=1.3pt] (1.351,0.390) -- (1.013,0.195);
\draw[black,line width=1.3pt] (1.013,0.195) -- (1.013,-0.195);
\draw[black,line width=1.3pt] (1.351,-0.390) -- (1.689,-0.195);
\draw[black,line width=1.3pt] (1.689,-0.195) -- (1.689,0.195);
\draw[black,line width=1.3pt] (0.675,-0.390) -- (0.338,-0.195);
\draw[black,line width=1.3pt] (0.000,-0.780) -- (0.338,-0.975);
\draw[black,line width=1.3pt] (0.338,-0.975) -- (0.675,-0.780);
\draw[black,line width=1.3pt] (-1.351,0.390) -- (-1.689,0.195);
\draw[black,line width=1.3pt] (-1.689,0.195) -- (-1.689,-0.195);
\draw[black,line width=1.3pt] (-0.675,-0.780) -- (-0.338,-0.975);
\draw[black,line width=1.3pt] (-0.338,-0.975) -- (-0.000,-0.780);
\end{tikzpicture}}
\par\vspace{2pt}{\footnotesize\textbf{Fig. 6.} Only exposed cell edges survive cancellation.\par}
\end{minipage}
\par\vspace{9pt}
\begin{minipage}[t]{0.48\linewidth}
\centering
\resizebox{0.96\linewidth}{!}{\begin{tikzpicture}[x=0.85cm,y=0.85cm,baseline=(current bounding box.center)]
\path[draw=gray!45,fill=cyan!17,line width=.45pt] (0.675,0.780) -- (0.338,0.975) -- (0.000,0.780) -- (0.000,0.390) -- (0.338,0.195) -- (0.675,0.390) -- cycle;
\path[draw=gray!45,fill=cyan!17,line width=.45pt] (1.013,-0.975) -- (0.675,-0.780) -- (0.338,-0.975) -- (0.338,-1.365) -- (0.675,-1.560) -- (1.013,-1.365) -- cycle;
\path[draw=gray!45,fill=cyan!17,line width=.45pt] (1.351,-0.390) -- (1.013,-0.195) -- (0.675,-0.390) -- (0.675,-0.780) -- (1.013,-0.975) -- (1.351,-0.780) -- cycle;
\path[draw=gray!45,fill=cyan!17,line width=.45pt] (4.053,0.780) -- (3.715,0.975) -- (3.377,0.780) -- (3.377,0.390) -- (3.715,0.195) -- (4.053,0.390) -- cycle;
\path[draw=gray!45,fill=cyan!17,line width=.45pt] (-0.338,-0.975) -- (-0.675,-0.780) -- (-1.013,-0.975) -- (-1.013,-1.365) -- (-0.675,-1.560) -- (-0.338,-1.365) -- cycle;
\path[draw=gray!45,fill=cyan!17,line width=.45pt] (-1.013,0.195) -- (-1.351,0.390) -- (-1.689,0.195) -- (-1.689,-0.195) -- (-1.351,-0.390) -- (-1.013,-0.195) -- cycle;
\path[draw=gray!45,fill=cyan!17,line width=.45pt] (-0.338,0.195) -- (-0.675,0.390) -- (-1.013,0.195) -- (-1.013,-0.195) -- (-0.675,-0.390) -- (-0.338,-0.195) -- cycle;
\path[draw=gray!45,fill=cyan!17,line width=.45pt] (1.013,1.365) -- (0.675,1.560) -- (0.338,1.365) -- (0.338,0.975) -- (0.675,0.780) -- (1.013,0.975) -- cycle;
\path[draw=gray!45,fill=cyan!17,line width=.45pt] (1.013,0.195) -- (0.675,0.390) -- (0.338,0.195) -- (0.338,-0.195) -- (0.675,-0.390) -- (1.013,-0.195) -- cycle;
\path[draw=gray!45,fill=cyan!17,line width=.45pt] (-0.338,1.365) -- (-0.675,1.560) -- (-1.013,1.365) -- (-1.013,0.975) -- (-0.675,0.780) -- (-0.338,0.975) -- cycle;
\path[draw=gray!45,fill=cyan!17,line width=.45pt] (-0.675,-0.390) -- (-1.013,-0.195) -- (-1.351,-0.390) -- (-1.351,-0.780) -- (-1.013,-0.975) -- (-0.675,-0.780) -- cycle;
\path[draw=gray!45,fill=cyan!17,line width=.45pt] (-0.675,0.780) -- (-1.013,0.975) -- (-1.351,0.780) -- (-1.351,0.390) -- (-1.013,0.195) -- (-0.675,0.390) -- cycle;
\path[draw=gray!45,fill=cyan!17,line width=.45pt] (0.000,0.780) -- (-0.338,0.975) -- (-0.675,0.780) -- (-0.675,0.390) -- (-0.338,0.195) -- (-0.000,0.390) -- cycle;
\path[draw=gray!45,fill=cyan!17,line width=.45pt] (1.351,0.780) -- (1.013,0.975) -- (0.675,0.780) -- (0.675,0.390) -- (1.013,0.195) -- (1.351,0.390) -- cycle;
\path[draw=gray!45,fill=cyan!17,line width=.45pt] (1.689,0.195) -- (1.351,0.390) -- (1.013,0.195) -- (1.013,-0.195) -- (1.351,-0.390) -- (1.689,-0.195) -- cycle;
\path[draw=gray!45,fill=cyan!17,line width=.45pt] (0.338,-0.975) -- (0.000,-0.780) -- (-0.338,-0.975) -- (-0.338,-1.365) -- (-0.000,-1.560) -- (0.338,-1.365) -- cycle;
\path[draw=gray!45,fill=cyan!17,line width=.45pt] (0.675,-0.390) -- (0.338,-0.195) -- (0.000,-0.390) -- (0.000,-0.780) -- (0.338,-0.975) -- (0.675,-0.780) -- cycle;
\path[draw=gray!45,fill=cyan!17,line width=.45pt] (3.715,0.195) -- (3.377,0.390) -- (3.040,0.195) -- (3.040,-0.195) -- (3.377,-0.390) -- (3.715,-0.195) -- cycle;
\path[draw=gray!45,fill=cyan!17,line width=.45pt] (4.391,0.195) -- (4.053,0.390) -- (3.715,0.195) -- (3.715,-0.195) -- (4.053,-0.390) -- (4.391,-0.195) -- cycle;
\path[draw=gray!45,fill=cyan!17,line width=.45pt] (0.338,1.365) -- (0.000,1.560) -- (-0.338,1.365) -- (-0.338,0.975) -- (-0.000,0.780) -- (0.338,0.975) -- cycle;
\path[draw=gray!45,fill=cyan!17,line width=.45pt] (0.000,-0.390) -- (-0.338,-0.195) -- (-0.675,-0.390) -- (-0.675,-0.780) -- (-0.338,-0.975) -- (-0.000,-0.780) -- cycle;
\draw[black,line width=1.3pt] (0.000,0.390) -- (0.338,0.195);
\draw[black,line width=1.3pt] (0.338,-1.365) -- (0.675,-1.560);
\draw[black,line width=1.3pt] (0.675,-1.560) -- (1.013,-1.365);
\draw[black,line width=1.3pt] (1.013,-1.365) -- (1.013,-0.975);
\draw[black,line width=1.3pt] (1.013,-0.975) -- (1.351,-0.780);
\draw[black,line width=1.3pt] (1.351,-0.780) -- (1.351,-0.390);
\draw[black,line width=1.3pt] (4.053,0.780) -- (3.715,0.975);
\draw[black,line width=1.3pt] (3.715,0.975) -- (3.377,0.780);
\draw[black,line width=1.3pt] (3.377,0.780) -- (3.377,0.390);
\draw[black,line width=1.3pt] (4.053,0.390) -- (4.053,0.780);
\draw[black,line width=1.3pt] (-1.013,-0.975) -- (-1.013,-1.365);
\draw[black,line width=1.3pt] (-1.013,-1.365) -- (-0.675,-1.560);
\draw[black,line width=1.3pt] (-0.675,-1.560) -- (-0.338,-1.365);
\draw[black,line width=1.3pt] (-1.351,0.390) -- (-1.689,0.195);
\draw[black,line width=1.3pt] (-1.689,0.195) -- (-1.689,-0.195);
\draw[black,line width=1.3pt] (-1.689,-0.195) -- (-1.351,-0.390);
\draw[black,line width=1.3pt] (-0.338,-0.195) -- (-0.338,0.195);
\draw[black,line width=1.3pt] (1.013,1.365) -- (0.675,1.560);
\draw[black,line width=1.3pt] (0.675,1.560) -- (0.338,1.365);
\draw[black,line width=1.3pt] (1.013,0.975) -- (1.013,1.365);
\draw[black,line width=1.3pt] (0.338,0.195) -- (0.338,-0.195);
\draw[black,line width=1.3pt] (-0.338,1.365) -- (-0.675,1.560);
\draw[black,line width=1.3pt] (-0.675,1.560) -- (-1.013,1.365);
\draw[black,line width=1.3pt] (-1.013,1.365) -- (-1.013,0.975);
\draw[black,line width=1.3pt] (-1.351,-0.390) -- (-1.351,-0.780);
\draw[black,line width=1.3pt] (-1.351,-0.780) -- (-1.013,-0.975);
\draw[black,line width=1.3pt] (-1.013,0.975) -- (-1.351,0.780);
\draw[black,line width=1.3pt] (-1.351,0.780) -- (-1.351,0.390);
\draw[black,line width=1.3pt] (-0.338,0.195) -- (-0.000,0.390);
\draw[black,line width=1.3pt] (1.351,0.780) -- (1.013,0.975);
\draw[black,line width=1.3pt] (1.351,0.390) -- (1.351,0.780);
\draw[black,line width=1.3pt] (1.689,0.195) -- (1.351,0.390);
\draw[black,line width=1.3pt] (1.351,-0.390) -- (1.689,-0.195);
\draw[black,line width=1.3pt] (1.689,-0.195) -- (1.689,0.195);
\draw[black,line width=1.3pt] (-0.338,-1.365) -- (-0.000,-1.560);
\draw[black,line width=1.3pt] (-0.000,-1.560) -- (0.338,-1.365);
\draw[black,line width=1.3pt] (0.338,-0.195) -- (0.000,-0.390);
\draw[black,line width=1.3pt] (3.377,0.390) -- (3.040,0.195);
\draw[black,line width=1.3pt] (3.040,0.195) -- (3.040,-0.195);
\draw[black,line width=1.3pt] (3.040,-0.195) -- (3.377,-0.390);
\draw[black,line width=1.3pt] (3.377,-0.390) -- (3.715,-0.195);
\draw[black,line width=1.3pt] (4.391,0.195) -- (4.053,0.390);
\draw[black,line width=1.3pt] (3.715,-0.195) -- (4.053,-0.390);
\draw[black,line width=1.3pt] (4.053,-0.390) -- (4.391,-0.195);
\draw[black,line width=1.3pt] (4.391,-0.195) -- (4.391,0.195);
\draw[black,line width=1.3pt] (0.338,1.365) -- (0.000,1.560);
\draw[black,line width=1.3pt] (0.000,1.560) -- (-0.338,1.365);
\draw[black,line width=1.3pt] (0.000,-0.390) -- (-0.338,-0.195);
\end{tikzpicture}}
\par\vspace{2pt}{\footnotesize\textbf{Fig. 7.} Disconnected components and a central hole.\par}
\end{minipage}
\hfill
\begin{minipage}[t]{0.48\linewidth}
\centering
\resizebox{0.96\linewidth}{!}{\begin{tikzpicture}[x=0.85cm,y=0.85cm,baseline=(current bounding box.center)]
\path[draw=gray!45,fill=cyan!17,line width=.45pt] (0.675,0.780) -- (0.338,0.975) -- (0.000,0.780) -- (0.000,0.390) -- (0.338,0.195) -- (0.675,0.390) -- cycle;
\path[draw=gray!45,fill=cyan!17,line width=.45pt] (1.013,-0.975) -- (0.675,-0.780) -- (0.338,-0.975) -- (0.338,-1.365) -- (0.675,-1.560) -- (1.013,-1.365) -- cycle;
\path[draw=gray!45,fill=cyan!17,line width=.45pt] (1.351,-0.390) -- (1.013,-0.195) -- (0.675,-0.390) -- (0.675,-0.780) -- (1.013,-0.975) -- (1.351,-0.780) -- cycle;
\path[draw=gray!45,fill=cyan!17,line width=.45pt] (4.053,0.780) -- (3.715,0.975) -- (3.377,0.780) -- (3.377,0.390) -- (3.715,0.195) -- (4.053,0.390) -- cycle;
\path[draw=gray!45,fill=cyan!17,line width=.45pt] (-0.338,-0.975) -- (-0.675,-0.780) -- (-1.013,-0.975) -- (-1.013,-1.365) -- (-0.675,-1.560) -- (-0.338,-1.365) -- cycle;
\path[draw=gray!45,fill=cyan!17,line width=.45pt] (-1.013,0.195) -- (-1.351,0.390) -- (-1.689,0.195) -- (-1.689,-0.195) -- (-1.351,-0.390) -- (-1.013,-0.195) -- cycle;
\path[draw=gray!45,fill=cyan!17,line width=.45pt] (-0.338,0.195) -- (-0.675,0.390) -- (-1.013,0.195) -- (-1.013,-0.195) -- (-0.675,-0.390) -- (-0.338,-0.195) -- cycle;
\path[draw=gray!45,fill=cyan!17,line width=.45pt] (1.013,1.365) -- (0.675,1.560) -- (0.338,1.365) -- (0.338,0.975) -- (0.675,0.780) -- (1.013,0.975) -- cycle;
\path[draw=gray!45,fill=cyan!17,line width=.45pt] (1.013,0.195) -- (0.675,0.390) -- (0.338,0.195) -- (0.338,-0.195) -- (0.675,-0.390) -- (1.013,-0.195) -- cycle;
\path[draw=gray!45,fill=cyan!17,line width=.45pt] (-0.338,1.365) -- (-0.675,1.560) -- (-1.013,1.365) -- (-1.013,0.975) -- (-0.675,0.780) -- (-0.338,0.975) -- cycle;
\path[draw=gray!45,fill=cyan!17,line width=.45pt] (-0.675,-0.390) -- (-1.013,-0.195) -- (-1.351,-0.390) -- (-1.351,-0.780) -- (-1.013,-0.975) -- (-0.675,-0.780) -- cycle;
\path[draw=gray!45,fill=cyan!17,line width=.45pt] (-0.675,0.780) -- (-1.013,0.975) -- (-1.351,0.780) -- (-1.351,0.390) -- (-1.013,0.195) -- (-0.675,0.390) -- cycle;
\path[draw=gray!45,fill=cyan!17,line width=.45pt] (0.000,0.780) -- (-0.338,0.975) -- (-0.675,0.780) -- (-0.675,0.390) -- (-0.338,0.195) -- (-0.000,0.390) -- cycle;
\path[draw=gray!45,fill=cyan!17,line width=.45pt] (1.351,0.780) -- (1.013,0.975) -- (0.675,0.780) -- (0.675,0.390) -- (1.013,0.195) -- (1.351,0.390) -- cycle;
\path[draw=gray!45,fill=cyan!17,line width=.45pt] (1.689,0.195) -- (1.351,0.390) -- (1.013,0.195) -- (1.013,-0.195) -- (1.351,-0.390) -- (1.689,-0.195) -- cycle;
\path[draw=gray!45,fill=cyan!17,line width=.45pt] (0.338,-0.975) -- (0.000,-0.780) -- (-0.338,-0.975) -- (-0.338,-1.365) -- (-0.000,-1.560) -- (0.338,-1.365) -- cycle;
\path[draw=gray!45,fill=cyan!17,line width=.45pt] (0.675,-0.390) -- (0.338,-0.195) -- (0.000,-0.390) -- (0.000,-0.780) -- (0.338,-0.975) -- (0.675,-0.780) -- cycle;
\path[draw=gray!45,fill=cyan!17,line width=.45pt] (3.715,0.195) -- (3.377,0.390) -- (3.040,0.195) -- (3.040,-0.195) -- (3.377,-0.390) -- (3.715,-0.195) -- cycle;
\path[draw=gray!45,fill=cyan!17,line width=.45pt] (4.391,0.195) -- (4.053,0.390) -- (3.715,0.195) -- (3.715,-0.195) -- (4.053,-0.390) -- (4.391,-0.195) -- cycle;
\path[draw=gray!45,fill=cyan!17,line width=.45pt] (0.338,1.365) -- (0.000,1.560) -- (-0.338,1.365) -- (-0.338,0.975) -- (-0.000,0.780) -- (0.338,0.975) -- cycle;
\path[draw=gray!45,fill=cyan!17,line width=.45pt] (0.000,-0.390) -- (-0.338,-0.195) -- (-0.675,-0.390) -- (-0.675,-0.780) -- (-0.338,-0.975) -- (-0.000,-0.780) -- cycle;
\draw[red!80!black,line width=1.3pt] (0.000,0.390) -- (0.338,0.195);
\draw[red!80!black,line width=1.3pt] (0.338,-1.365) -- (0.675,-1.560);
\draw[red!80!black,line width=1.3pt] (0.675,-1.560) -- (1.013,-1.365);
\draw[red!80!black,line width=1.3pt] (1.013,-1.365) -- (1.013,-0.975);
\draw[red!80!black,line width=1.3pt] (1.013,-0.975) -- (1.351,-0.780);
\draw[red!80!black,line width=1.3pt] (1.351,-0.780) -- (1.351,-0.390);
\draw[red!80!black,line width=1.3pt] (4.053,0.780) -- (3.715,0.975);
\draw[red!80!black,line width=1.3pt] (3.715,0.975) -- (3.377,0.780);
\draw[red!80!black,line width=1.3pt] (3.377,0.780) -- (3.377,0.390);
\draw[red!80!black,line width=1.3pt] (4.053,0.390) -- (4.053,0.780);
\draw[red!80!black,line width=1.3pt] (-1.013,-0.975) -- (-1.013,-1.365);
\draw[red!80!black,line width=1.3pt] (-1.013,-1.365) -- (-0.675,-1.560);
\draw[red!80!black,line width=1.3pt] (-0.675,-1.560) -- (-0.338,-1.365);
\draw[red!80!black,line width=1.3pt] (-1.351,0.390) -- (-1.689,0.195);
\draw[red!80!black,line width=1.3pt] (-1.689,0.195) -- (-1.689,-0.195);
\draw[red!80!black,line width=1.3pt] (-1.689,-0.195) -- (-1.351,-0.390);
\draw[red!80!black,line width=1.3pt] (-0.338,-0.195) -- (-0.338,0.195);
\draw[red!80!black,line width=1.3pt] (1.013,1.365) -- (0.675,1.560);
\draw[red!80!black,line width=1.3pt] (0.675,1.560) -- (0.338,1.365);
\draw[red!80!black,line width=1.3pt] (1.013,0.975) -- (1.013,1.365);
\draw[red!80!black,line width=1.3pt] (0.338,0.195) -- (0.338,-0.195);
\draw[red!80!black,line width=1.3pt] (-0.338,1.365) -- (-0.675,1.560);
\draw[red!80!black,line width=1.3pt] (-0.675,1.560) -- (-1.013,1.365);
\draw[red!80!black,line width=1.3pt] (-1.013,1.365) -- (-1.013,0.975);
\draw[red!80!black,line width=1.3pt] (-1.351,-0.390) -- (-1.351,-0.780);
\draw[red!80!black,line width=1.3pt] (-1.351,-0.780) -- (-1.013,-0.975);
\draw[red!80!black,line width=1.3pt] (-1.013,0.975) -- (-1.351,0.780);
\draw[red!80!black,line width=1.3pt] (-1.351,0.780) -- (-1.351,0.390);
\draw[red!80!black,line width=1.3pt] (-0.338,0.195) -- (-0.000,0.390);
\draw[red!80!black,line width=1.3pt] (1.351,0.780) -- (1.013,0.975);
\draw[red!80!black,line width=1.3pt] (1.351,0.390) -- (1.351,0.780);
\draw[red!80!black,line width=1.3pt] (1.689,0.195) -- (1.351,0.390);
\draw[red!80!black,line width=1.3pt] (1.351,-0.390) -- (1.689,-0.195);
\draw[red!80!black,line width=1.3pt] (1.689,-0.195) -- (1.689,0.195);
\draw[red!80!black,line width=1.3pt] (-0.338,-1.365) -- (-0.000,-1.560);
\draw[red!80!black,line width=1.3pt] (-0.000,-1.560) -- (0.338,-1.365);
\draw[red!80!black,line width=1.3pt] (0.338,-0.195) -- (0.000,-0.390);
\draw[red!80!black,line width=1.3pt] (3.377,0.390) -- (3.040,0.195);
\draw[red!80!black,line width=1.3pt] (3.040,0.195) -- (3.040,-0.195);
\draw[red!80!black,line width=1.3pt] (3.040,-0.195) -- (3.377,-0.390);
\draw[red!80!black,line width=1.3pt] (3.377,-0.390) -- (3.715,-0.195);
\draw[red!80!black,line width=1.3pt] (4.391,0.195) -- (4.053,0.390);
\draw[red!80!black,line width=1.3pt] (3.715,-0.195) -- (4.053,-0.390);
\draw[red!80!black,line width=1.3pt] (4.053,-0.390) -- (4.391,-0.195);
\draw[red!80!black,line width=1.3pt] (4.391,-0.195) -- (4.391,0.195);
\draw[red!80!black,line width=1.3pt] (0.338,1.365) -- (0.000,1.560);
\draw[red!80!black,line width=1.3pt] (0.000,1.560) -- (-0.338,1.365);
\draw[red!80!black,line width=1.3pt] (0.000,-0.390) -- (-0.338,-0.195);
\end{tikzpicture}}
\par\vspace{2pt}{\footnotesize\textbf{Fig. 8.} All surviving edges form closed contour cycles.\par}
\end{minipage}
\par\vspace{9pt}
\endgroup
\appendix
\section{Derivation of all six local formulas}\label{app:formulas}
Set $y=2k$ or $y=2k+1$ in~\eqref{eq:forward}, substitute the five triples~\eqref{eq:aux}, normalize each, and apply~\eqref{eq:reverse}. In the ordinary map-index range, the result is:
\begin{table}[h!]\centering\small
\caption{Even-row formulas, $y=2k$. Every entry is an integer pair.}\label{tab:even}
\begin{tabular}{ccc}\toprule Record & Point $P_i$ & Tail $T_i$\\\midrule
$e_1$&$(x,y-1)$&$(x,y+1)$\\
$e_2$&$(x,y)$&$(x,y+1)$\\
$e_3$&$(x,y)$&$(x-1,y+1)$\\
$e_4$&$(x-1,y-1)$&$(x-1,y+1)$\\
$e_5$&$(x-1,y-1)$&$(x,y)$\\
$e_6$&$(x,y-1)$&$(x,y)$\\\bottomrule
\end{tabular}
\end{table}
\begin{table}[h!]\centering\small
\caption{Odd-row formulas, $y=2k+1$.}\label{tab:odd}
\begin{tabular}{ccc}\toprule Record & Point $P_i$ & Tail $T_i$\\\midrule
$e_1$&$(x+1,y-1)$&$(x+1,y+1)$\\
$e_2$&$(x,y)$&$(x+1,y+1)$\\
$e_3$&$(x,y)$&$(x,y+1)$\\
$e_4$&$(x,y-1)$&$(x,y+1)$\\
$e_5$&$(x,y-1)$&$(x,y)$\\
$e_6$&$(x+1,y-1)$&$(x,y)$\\\bottomrule
\end{tabular}
\end{table}

For illustration, if $y=2k$, then $A=x+k+1$ and $B=2k+1$. Applying $D$ to $U_2=(A+1,B,1)$ requires normalization by one, giving $(A,B-1,0)$; thus $D(U_2)=(x,y-1)$. Applying $D$ to $U_1=(A+1,B+1,0)$ gives $(x,y+1)$. Hence $e_1=((x,y-1),(x,y+1))$. The remaining entries follow by the same substitutions, and no geometric approximation is involved.

\subsection{Complete five-triple evaluation}
Write $D_i=D(U_i)$. Direct substitution, using $D(a,b,c)=
(a-m)-\lfloor((b-m)+(c-m))/2\rfloor-1,\ b-c-1$ with
$m=\min(a,b,c)$, yields the following \emph{five} coordinate
identifiers. The three point-role vertices are $D_2,D_0,D_4$;
the three tail-role vertices are $D_1,D_3,D_0$.
\begin{center}\small
\begin{tabular}{ccll}\toprule
$i$ & triple & $y=2k$ & $y=2k+1$\\\midrule
0 & $(A,B,0)$ & $(x,y)$ & $(x,y)$\\
1 & $(A+1,B+1,0)$ & $(x,y+1)$ & $(x+1,y+1)$\\
2 & $(A+1,B,1)$ & $(x,y-1)$ & $(x+1,y-1)$\\
3 & $(A,B+1,0)$ & $(x-1,y+1)$ & $(x,y+1)$\\
4 & $(A,B,1)$ & $(x-1,y-1)$ & $(x,y-1)$\\\bottomrule
\end{tabular}
\end{center}
\noindent\textbf{Derivation.}
For $y=2k$ the unnormalized input is $(A,B,0)=(x+k+1,2k+1,0)$.
The normalized inverse map is invariant under common integer translations:
for every integer $t$,
\[
 D(a+t,b+t,c+t)=D(a,b,c),
\]
because normalization removes the same common shift. Consequently one
may compute $D$ using any representative and the difference $b-c$.
For the five triples, the respective differences $b-c$ are
$2k+1,2k+2,2k,2k+2,2k$; the corresponding first coordinates
simplify to $x,x,x,x-1,x-1$.
For $y=2k+1$, the input is $(x+k+2,2k+2,0)$; the five
$b-c$ values are $2k+2,2k+3,2k+1,2k+3,2k+1$ and their first
coordinates are $x,x+1,x+1,x,x$. These calculations give all
five rows above, and insertion into the six pairs in
\eqref{eq:records} gives Tables~\ref{tab:even}--\ref{tab:odd}.
All operations here are mathematical floor operations. The C implementation
uses truncating integer division; the two coincide for its normal
nonnegative map-index inputs but should not be conflated on arbitrary
negative C inputs.

\begin{lemma}[No within-role identification]
For every integer $(x,y)$ the point identifiers $D_2,D_0,D_4$
are pairwise distinct, as are the tail identifiers $D_1,D_3,D_0$.
\end{lemma}
\begin{proof}
In either parity, the point identifiers have second coordinates
$y-1,y,y-1$; the two with equal second coordinate differ by one
in the first. The tail identifiers have second coordinates
$y+1,y+1,y$; the two with equal second coordinate differ by one
in the first. Hence neither role has duplicate vertices.
\end{proof}

\section{Full endpoint-incidence tables}\label{app:vertices}
The notation $[i,j]$ means that a cell contributes local edges $e_i,e_j$ at the given role-qualified vertex. The following tables are obtained by solving the endpoint equations in Appendix~\ref{app:formulas} for the cell label.
\begin{center}\small
\begin{tabular}{llll}\toprule
Role and parity of $v$ & Cell 1 & Cell 2 & Cell 3\\\midrule
$\PP(u,v)$, $v$ even & $(u,v):[2,3]$ & $(u,v+1):[4,5]$ & $(u-1,v+1):[1,6]$\\
$\PP(u,v)$, $v$ odd & $(u,v):[2,3]$ & $(u,v+1):[1,6]$ & $(u+1,v+1):[4,5]$\\
$\TT(u,v)$, $v$ even & $(u,v):[5,6]$ & $(u,v-1):[3,4]$ & $(u-1,v-1):[1,2]$\\
$\TT(u,v)$, $v$ odd & $(u,v):[5,6]$ & $(u,v-1):[1,2]$ & $(u+1,v-1):[3,4]$\\\bottomrule
\end{tabular}
\end{center}
Each row lists exactly three cells, each contributing two incident local edges. Across the three cells, the six contributions pair into exactly three distinct shared records. This proves degree three in the \emph{full} incidence graph. After selecting a finite cell set and cancelling shared records, the degree becomes zero or two, as shown in Theorem~\ref{thm:degree}.

\subsection{Exhaustiveness and three-edge incidence}
The endpoint lists can be derived without geometric assumptions.
For example, fix $\PP(u,v)$ with even $v$. Its second coordinate
can be either $y$ (from the $e_2,e_3$ point), or $y-1$
(from $e_1,e_6$ or $e_4,e_5$). In the first case $y=v$ and
$x=u$. In the second case $y=v+1$ is odd; solving the two
odd-row first-coordinate equations gives respectively
$x=u-1$ and $x=u$. Thus exactly the three cells displayed
in the first row of the table are possible. The other three
rows follow by the same two-case inversion.

For each role-qualified vertex, the three cells contribute two
edges each. The following table gives the \emph{three pairs of
identical records}, using the cell labels $F_1,F_2,F_3$ in the
order of the preceding table. An entry $F_a:i=F_b:j$
means equality of the complete ordered record, not merely equality
of its endpoint at the vertex.
\begin{center}\small
\begin{tabular}{p{30mm}p{103mm}}\toprule
Vertex case & Paired local edges \\\midrule
$\PP(u,v)$, $v$ even & $F_1:2=F_2:5$; $F_1:3=F_3:6$; $F_2:4=F_3:1$\\
$\PP(u,v)$, $v$ odd & $F_1:2=F_3:5$; $F_1:3=F_2:6$; $F_2:1=F_3:4$\\
$\TT(u,v)$, $v$ even & $F_1:5=F_3:2$; $F_1:6=F_2:3$; $F_2:4=F_3:1$\\
$\TT(u,v)$, $v$ odd & $F_1:5=F_2:2$; $F_1:6=F_3:3$; $F_2:1=F_3:4$\\\bottomrule
\end{tabular}
\end{center}
The three paired records are distinct: in each cell its two incident
edges have distinct other endpoints, and the three pairs form a
triangle of cell adjacencies. Thus each vertex has exactly three
full-lattice incident edges, each generated by exactly two cells.
This establishes both edge uniqueness and endpoint-incidence
preservation from the explicit integer equations.

\section{All twelve neighbor identities}\label{app:neighbors}
The identities below abbreviate $e_i(x,y)=e_j(x+\delta_x,y+\delta_y)$; each is checked by substituting Tables~\ref{tab:even} and~\ref{tab:odd}.
\begin{center}\small
\begin{tabular}{cll}\toprule
$i$&Even $y$&Odd $y$\\\midrule
1&$e_1(x,y)=e_4(x+1,y)$&$e_1(x,y)=e_4(x+1,y)$\\
2&$e_2(x,y)=e_5(x,y+1)$&$e_2(x,y)=e_5(x+1,y+1)$\\
3&$e_3(x,y)=e_6(x-1,y+1)$&$e_3(x,y)=e_6(x,y+1)$\\
4&$e_4(x,y)=e_1(x-1,y)$&$e_4(x,y)=e_1(x-1,y)$\\
5&$e_5(x,y)=e_2(x-1,y-1)$&$e_5(x,y)=e_2(x,y-1)$\\
6&$e_6(x,y)=e_3(x,y-1)$&$e_6(x,y)=e_3(x+1,y-1)$\\\bottomrule
\end{tabular}
\end{center}
Together with Appendix~\ref{app:vertices}, these equalities rule out accidental matching with non-neighbors: the endpoints determine one of the three edge types at a point-role vertex, and each such type occurs in exactly two incident cells.

\subsection{Parity-complete substitution certificate}
For completeness, the following table records the two endpoints
$(P;T)$ of every even-row edge and the partner obtained by
substitution in the neighboring row. The equality is literal
integer-pair equality; $x,y$ are arbitrary integers in the
mathematical parity extension.
\begin{center}\scriptsize
\begin{tabular}{clll}\toprule
$i$ & central even-row $(P;T)$ & partner cell & partner edge \\\midrule
1 & $(x,y-1);(x,y+1)$ & $(x+1,y)$ & $e_4$\\
2 & $(x,y);(x,y+1)$ & $(x,y+1)$ & $e_5$\\
3 & $(x,y);(x-1,y+1)$ & $(x-1,y+1)$ & $e_6$\\
4 & $(x-1,y-1);(x-1,y+1)$ & $(x-1,y)$ & $e_1$\\
5 & $(x-1,y-1);(x,y)$ & $(x-1,y-1)$ & $e_2$\\
6 & $(x,y-1);(x,y)$ & $(x,y-1)$ & $e_3$\\\bottomrule
\end{tabular}
\end{center}
\begin{center}\scriptsize
\begin{tabular}{clll}\toprule
$i$ & central odd-row $(P;T)$ & partner cell & partner edge \\\midrule
1 & $(x+1,y-1);(x+1,y+1)$ & $(x+1,y)$ & $e_4$\\
2 & $(x,y);(x+1,y+1)$ & $(x+1,y+1)$ & $e_5$\\
3 & $(x,y);(x,y+1)$ & $(x,y+1)$ & $e_6$\\
4 & $(x,y-1);(x,y+1)$ & $(x-1,y)$ & $e_1$\\
5 & $(x,y-1);(x,y)$ & $(x,y-1)$ & $e_2$\\
6 & $(x+1,y-1);(x,y)$ & $(x+1,y-1)$ & $e_3$\\\bottomrule
\end{tabular}
\end{center}
Every partner equality follows by inserting the indicated neighbor
coordinates into the \emph{opposite-parity} or same-parity row
of Appendix~\ref{app:formulas}. Together with the exhaustive
vertex-inversion argument in Appendix~\ref{app:vertices}, these
twelve substitutions prove that each record is generated by
exactly two, and never three, cells.

\section{Reproducible computations, verification, and limits}
\label{app:verification}
This appendix separates \emph{symbolic proof} from finite computational
checks. The former establishes the statements for every admissible
integer cell label; the latter detects transcription mistakes and
tests the cycle-count consequences on finite inputs.

\subsection{Reproducible program}
The archive accompanying this manuscript contains
\texttt{verify\_appendices.py}, an executable Python 3 program
that independently implements (i) normalization and the five
A--B--C triples, (ii) the parity-expanded record formulas,
(iii) neighbor matching, (iv) full-lattice endpoint incidence,
and (v) cycle extraction and an independent selected-cell/complement
component count. It requires only the Python standard library.
Run it with
\begin{center}\texttt{python3 verify\_appendices.py}\end{center}
The deterministic pseudorandom generator is initialized with seed
\texttt{20261018}.

\subsection{Finite-check results}
The verification program completed the following checks:
\begin{center}\small
\begin{tabular}{p{100mm}r}\toprule
Check & Count \\\midrule
Cells comparing direct A--B--C records with parity formulas & 5,041\\
Individual neighbor edge equalities & 29,394\\
Interior cell-edge incidences tested for multiplicity and valence & 21,960\\
Finite selected-cell sets checked for cycle count $m=c+h$ & 504\\\bottomrule
\end{tabular}
\end{center}
The cell tests include negative $x$ indices; the direct A--B--C
conversion is tested on $0\le y\le70$, and the incidence tests
avoid the finite-window boundary. The 504 finite sets include the
empty set, singleton and disconnected examples, a filled rectangle,
and 500 seeded pseudorandom subsets of an $8\times8$ window.
For each set, the program independently counts connected components
in the selected cell-adjacency graph and bounded connected components
of the complementary cell-adjacency graph in a padded rectangle.
Every tested set satisfied $m=c+h$. These are finite checks, not
an exhaustive verification of all finite subsets.

\subsection{Scope of the established result}
\begin{enumerate}[label=\textbf{\Alph*.}]
\item \textbf{Symbolic coordinate identities.} The five-triple
      table and twelve neighbor identities are algebraic consequences
      of the stated floor-based formulas for all integer $x,y$.
\item \textbf{Combinatorial incidence.} The endpoint inversion
      and paired-edge table prove exactly three incident cells
      and three incident edges per role-qualified vertex.
\item \textbf{Topology.} Edge cancellation and degree-two incidence
      prove the complete cycle decomposition for every finite set.
\item \textbf{Algorithm.} Expected linear complexity concerns the
      indexed dictionary implementation, not the original C scans.
\item \textbf{Source arithmetic.} The original C forward conversion
      uses integer division and differs from mathematical floor for
      some negative indices. The source-level theorem is asserted on
      its ordinary nonnegative map-index domain; the all-integer
      formulas define a separate mathematical extension.
\item \textbf{Implementation limits.} The C routine's fixed array
      capacities, sentinel values, and closure conventions require
      separate software verification. This note does not claim to
      establish their unrestricted correctness.
\end{enumerate}

\section*{Code and reproducibility}\begingroup\sloppy
The original C source excerpts and independent Python verification programs are publicly available at
\url{https://github.com/Filuaou/hexagones-abc-boundaries}.
The files \path{DualCoordinateTransforms.inc.c} and \path{MovementBoundaryTopology.inc.c} are historical source excerpts, not a self-contained C application. The scripts \path{verify_appendices.py} and \path{compare_encodings.py} independently test the mathematical coordinate records and the cube-coordinate incidence comparison; they do not establish unrestricted correctness of the original fixed-capacity game routine. The repository contains code and verification materials; this article is intended for separate distribution through arXiv.

\endgroup\section*{Acknowledgements and disclosure of AI assistance}
The author acknowledges the use of ChatGPT (OpenAI) for assistance with source-code inspection, organization and refactoring of the surrounding project, mathematical exposition, verification-script development, and preparation of the LaTeX/TikZ manuscript. The original C algorithms and coordinate formulas were not altered as part of the mathematical presentation. The author remains responsible for checking the correctness of the proofs, citations, code, and final manuscript. This acknowledgement does not attribute scientific authorship to the AI system.

\section*{Source note}
The coordinate and record formulas were transcribed from the active Increment 38 project source, \path{DualCoordinateTransforms.inc.c} and \path{MovementBoundaryTopology.inc.c}. The historical source excerpts and verification scripts are available in the public code repository cited above; they are not required to compile the arXiv manuscript. Independent mathematical review is recommended.
\end{document}